\documentclass[preprint,12pt]{elsarticle}

\usepackage{amssymb}
\usepackage{amsthm}
\usepackage{amsmath}
\usepackage{graphics}
\usepackage{tikz}

\newtheorem{theorem}{Theorem}[section]
\newtheorem{lemma}[theorem]{Lemma}

\newtheorem{proposition}[theorem]{Proposition}

\newtheorem{remark}[theorem]{Remark}
\begin{document}
\begin{frontmatter}
\title{On the Long Time Existence of a Fractional KdV-BBM Type Equation}
\author{Goksu Oruc \corref{cor1}}
\ead{goeksu.oruc@uni-ulm.de}

\address{University of Ulm, Institute of Numerical Mathematics, Ulm,
	     Germany.}
 
\begin{abstract}
We consider a fractional Korteweg de Vries-Benjamin Bona Mahony (KdV-BBM) type equation including both fractional dispersive terms of  fractional KdV and fractional BBM equations. We aim to enhance the existence time of solutions with small initial data $|| u_0||_{H^{N+\alpha/2}}= \epsilon$  from $\frac{1}{\epsilon}$ to  $\frac{1}{\epsilon^2}$. The proof relies on the combination of a modified energy method with Fourier techniques.  In addition, the long time existence issues are investigated numerically.  Numerical observations of the lifespan give an evidence of existence of solutions beyond the hyperbolic time scale. This study provides a detailed analysis from both analytical and numerical aspects for the existence of smooth solutions.
\end{abstract}


\end{frontmatter}
\renewcommand{\theequation}{\arabic{section}.\arabic{equation}}
\setcounter{equation}{0}
\section{Introduction}
In this work, we consider the Cauchy problem for a fractional KdV-BBM type equation
\begin{eqnarray} 
&& u_t+ \kappa u_x + \lambda uu_x  - \mu D^{\alpha} u_x + \nu D^{\alpha} u_t  = 0, \label{fKdVBBM} \\
&& u(0,x)=u_0(x) \label{u0},
\end{eqnarray} 
for $u(t,x): [0,T] \times \mathbb{R} \rightarrow \mathbb{R}$, where $\kappa, \mu \in \mathbb{R}$, $\lambda, \nu \in \mathbb{R}^{+}$ and $D^{\alpha} =(- \Delta)^{\frac{\alpha}{2}}$ denotes the pseudo-differential operator defined by using the Fourier transform $\mathcal{F}$ 
 $$\mathcal{F}({D^{\alpha}q})(k)=|k|^{\alpha}\hat{q}(k),$$
with $0 < \alpha < 1$. We assume that the initial data $u_0$ in \eqref{u0} is sufficiently small
$$ \|u_0 \|_{H^{N+\alpha/2}(\mathbb{R})} \leq \epsilon \ll 1.$$
Here, $H^{s}(\mathbb{R})$ is the standard Sobolev space.

The special cases of the fractional KdV-BBM type equation contain a large class of equations. When $\alpha=2$ the equation  \eqref{fKdVBBM} turns into the standard KdV-BBM equation, which was proposed in \cite{bona,fetecau} as a model equation for unidirectional long wave propagation in dispersive media. The global wellposedness of the solutions to the Cauchy problem was established in \cite{mancoz} for $u_0 \in H^s(\mathbb{R})$ with $s\geq 1$ and $\nu>0$. In addition, the qualitative behaviour of the solutions such as dispersive shock formation, solitary waves and their interactions were studied numerically in \cite{mitsotakis,kalisch2013,FRANCIUS,kamgahigh}. The solitonic gas behaviour was investigated numerically in \cite{dutykh}. The version of the equation \eqref{fKdVBBM} with $\mu=-\frac{3}{4}$ and $\nu=\frac{5}{4}$ was derived to model unidirectional small amplitude long wave propagation in elastic medium in \cite{hae,hae2}. Then, the local wellposedness of solutions was obtained in $ H^s(\mathbb{R})$ with $s\geq 2- \frac{\alpha}{2}$, when $\alpha<1$ in \cite{oruc}.  For $1\leq \alpha \leq 2$,  it was proved in \cite{fabio} that the Cauchy problem with initial data $u_0 \in H^{\alpha/2}(\mathbb{R})$ is globally wellposed. If we choose $\alpha=2$, equation \eqref{fKdVBBM} corresponds to standard KdV equation for $\nu=0$ and to standard BBM equation for $\mu=0$. Moreover, it is known as Benjamin-Ono (BO) equation when $\alpha=1$. These equations have globally wellposed solutions in appropriate Sobolev spaces $ H^s(\mathbb{R})$ \cite{saut,colliander,bona,bona2009,ponce,tao}. Besides the standard unidirectional model mentioned above, one can also obtain the fractional KdV equation by taking $\nu=0$ in the equation \eqref{fKdVBBM} and the fractional BBM equation for $\mu=0$. The fractional KdV equation was shown in \cite{molinet2025} to be globally well-posed in energy space $ H^{\frac{s}{2}}(\mathbb{R})$ for $s>\frac{4}{5}$. On the other hand, the local wellposedness of the Cauchy problem associated to fractional BBM equation was proved in $ H^s(\mathbb{R})$ for  $s>2-\frac{\alpha}{2}$ in \cite{he}. In both problems the fractional case $\alpha<1$ does not permit the use of conservation laws to globalize the solution. The numerical evidence in \cite{klein} suggests the global existence of a solution to the fractional KdV equation for $\alpha> \frac{1}{2}$ and for fractional BBM equation when $\alpha> \frac{1}{3}.$ Long time existence issues for both equations have been studied analytically in \cite{ehrns,nilsson} by using a modified energy method. With this technique, it was possible to extend the lifespans of equations with small initial data beyond the hyperbolic time scale. The analysis is carried out in the Sobolev spaces  $ H^N(\mathbb{R})$ with  $N\geq3$ for the fractional KdV equation and, in $H^{N+\frac{\alpha}{2}}(\mathbb{R})$ with  $N\geq2$  for the fractional BBM equation. From numerical point of view, the study \cite{dutykh2013geo} proposes a geometric integration method to observe long time dynamics of the KdV equation.  Moreover,
the study \cite{klein} addresses the numerical investigation of the large time behaviour of solutions to the fractional KdV and fractional BBM equations. 

A natural question arises here: How do the existence of both fractional dispersive terms namely, $D^{\alpha} u_x $ of the fractional KdV and $D^{\alpha} u_t$ of the fractional BBM in \eqref{fKdVBBM} affect the lifespan of the solutions?

A numerical study \cite{moldobayev} revealed that the solutions of the Pad\'e model with both linear terms of the KdV and BBM equations provide a better approximation to the solutions of the full Euler equations than either the KdV or BBM equations. 
Furthermore, existence of fractional derivative in the dispersive terms gives the opportunity to investigate the effects of lower dispersion on the solutions. These reasons make equation \eqref{fKdVBBM} physically relevant and mathematically interesting problem.

The aim of this study is to provide an analysis in order to extend the existence time of solutions to the Cauchy problem \eqref{fKdVBBM}-\eqref{u0} beyond the hyperbolic time scale $\frac{1}{\epsilon}$ to $\frac{1}{\epsilon^2}$ and to provide numerical investigations of the long time behaviour of these solutions.

Following \cite{hunter,ehrns,nilsson}, our main result reads.
\begin{theorem} \label{theorem}
For $0<\alpha <1$ and $N\geq 2,$ there exists $\epsilon_0>0$, such that for any initial data $u(0,x)=u_0(x)$ satisfying
$$ \|u_0 \|_{H^{N+\alpha/2}}(\mathbb{R})  \leq \epsilon \leq \epsilon_0, $$
there exists a unique solution $u \in {C}([0,T]; H^{N+\alpha/2}(\mathbb{R}))$ of \eqref{fKdVBBM}
with $T \gtrsim \frac{1}{\epsilon^2} $  and 
$$ \|u\|_{  C([0,T];H^{N+\alpha/2}(\mathbb{R})) } \lesssim \|u_0 \|_{H^{N+\alpha/2}(\mathbb{R})}.$$
\end{theorem}

The outline of the paper is as follows. The derivation of the fractional KdV-BBM equation \eqref{fKdVBBM} and its properties are given in Section $2$. The maximal existence time of the solution is extended by using a modified energy method in Section $3$. The long time existence issues for smooth solutions are investigated numerically in Section $4$.

\subsection*{Notations}
We will use following notations throughout this paper. A  Fourier transform for  
tempered distributions $u(x)$ denoted by $\mathcal{F}(u)(k)$ with dual variable $k$ 
and its inverse are defined by
\begin{align}
	\mathcal{F}(u)(k)& = \int_{\mathbb{R}}^{}u(x)e^{-ikx}dx,\quad k\in 
	\mathbb{R},
	\nonumber\\
	u(x) & =\frac{1}{2\pi}\int_{\mathbb{R}}^{}\mathcal{F}(u)(k)e^{ikx} 
	dk,\quad x\in\mathbb{R},
	\label{fourierdef}
\end{align}
respectively.
 $L^2(\mathbb{R})$ is the usual Lebesgue space with norm $\|.\|_{L^{2}}$ and inner product $\langle f, g \rangle = \int_{\mathbb{R}} f g \ dx$. If $\frac{f}{g}$ is uniformly bounded from above we will use $f \lesssim g$ and the notaion $ \simeq$ will be used if $f \lesssim g \lesssim f$. $\mathcal{C}_{\kappa,\lambda,\nu,\mu}$ and $\mathcal{\bar{C}}_{\kappa,\lambda,\nu,\mu}$ will denote the generic constants. 

\section{The Governing Equation}
\noindent
In this section, we provide a formal derivation of the fractional KdV-BBM equation with general constants and its properties. In the context of dispersive models, the studies \cite{ADD1,CD1} present the approaches to derive BBM-type equations that achieve both Galilean invariance and energy conservation. Moreover, the authors in \cite{CD2} proposed an analysis for deriving model equations with fractional operators that still conserve these symmetry properties. Here, we follow the approach in \cite{hae} for the derivation of the fractional KdV-BBM equation. In order to make the derivation more transperent, we first briefly recall assumptions, asymptotic relations and scaling details from \cite{hae} and then show suitable choice of coefficients to get the governing equation.

As a starting point,  the well-known fractional type improved Boussinesq equation
\begin{equation}\label{Boussinesq}
  u_{tt} -u_{xx} + D^{\alpha} u_{tt}= (u^2)_{xx},
\end{equation}
is considered. We assume unidirectional (right-going) small-but-finite amplitude long wave solutions to equation \eqref{Boussinesq}.  First, the scaling transformation is introduced as
\begin{align} \label{tr1}
u(x,t)=\varepsilon U(\rho(x-t),\rho t)= \varepsilon U(Y,S)
\end{align}
 with $Y= \rho(x-t),~ S=\rho t.$ Here, $\varepsilon$ and $\rho$ are a small parameters and measure nonlinear and dispersive effects, respectively. Substitution of the scale transformation \eqref{tr1} into equation \eqref{Boussinesq} yields 
\begin{align}\label{eqi}
U_{SS}-2U_{YS}+\rho^{\alpha} D^{\alpha}(U_{SS} + U_{YY}-2U_{SY})- \varepsilon(U^2)_{YY}=0.
\end{align} 
Now, one seeks a solution in the form of single asymptotic expansion
\begin{align} \label{asp}
U(Y,S;\varepsilon, \rho)=U_0(Y,S)+ \varepsilon U_1(Y,S) + \rho^{\alpha} U_2(Y,S)+ \mathcal{O}(\varepsilon^2, \varepsilon \rho^{\alpha}, \rho^{2\alpha}),
\end{align}
as $\varepsilon, ~\rho$ $\rightarrow$ $0$. Moreover, unknowns $U_n$ $(n=0,1,2...)$ and their derivatives decay to zero as $|Y|$  $\rightarrow$ $0$. Afterwards, derivation follows inserting \eqref{asp} into equation \eqref{eqi}. For the calculation of unknowns with their derivatives and further details we refer the reader to \cite{hae}. Then, one gets
\begin{align}\label{equ}
U_S+ \varepsilon U U_Y - \frac{\rho^{\alpha}}{2}D^{\alpha}U_Y=0,
\end{align}
with $U(Y,S;\varepsilon, \rho).$ It is worth noting here that we aim to derive an equation so that it covers an additional fractional dispersive term of the KdV-type. Therefore,  the moving frame is defined by $U(Y,S)=V(aY+bS, cS)=V(X,T)$, where
\begin{equation*}
  X= a Y + b S,~~~~ \quad T = c S. 
\end{equation*}
Here $a$, $b$ and $c$ are positive constants to be chosen appropriately later. Equation \eqref{equ} in new coordinate system is given by
\begin{align*}
V_T+ \frac{b}{c}V_X + \frac{\varepsilon a}{c} V V_X + \rho^{\alpha} \frac{a^3}{2b} D^{\alpha}V_T=0. 
\end{align*}
In order to remove parameters $\varepsilon$ and $\rho$ from above equation, following scaling transformation is introduced
\begin{equation*}\label{last}
  v= \varepsilon V, \quad X=\rho \xi, \quad   T=  \rho \tau.
\end{equation*}
Later, it becomes the well-known fractional BBM equation
\begin{equation}\label{fBBM1}
  v_{\tau} + \omega v_{\xi} + 3 v v_{\xi} + D^{\alpha} v_{\tau}=0,
\end{equation}
with $\omega=\frac{3}{2}a^{\alpha}$ \cite{kapitula}. 
Using the coordinate transformation between original frame $(x,t)$ and new frame  $(\xi,\tau)$ 
\begin{equation*}\label{star}
  \xi=a x + (b-a) t  ,~~~~\quad  \tau=  c t,
\end{equation*}
one can easily calculate
\begin{align*}
  &v_{\tau} = \frac{1}{c}{v_t} + \frac{a-b}{ac} v_x, \\
  &v_{\xi} = \frac{1}{a} v_x ,\\
  &D^{\alpha} v_{\tau} = \frac{1}{a^{\alpha}} D^{\alpha} \left( \frac{1}{c}v_t+ \frac{a-b}{ac}{u_x} \right).
\end{align*}
Now, equation \eqref{fBBM1} can be written in orginial refernce frame as follows
\begin{equation} \label{fBBM2}
  v_t + (\frac{a-b +\omega c}{a} ) v_x + \frac{3c}{a} v v_x+  \frac{1}{a^{\alpha}} D^{\alpha} [  \frac{a-b}{a} v_x + v_t ] =0.
\end{equation}
Equation \eqref{fBBM2} can be simplified with an appropriate change of variables  
\begin{equation*}
 \frac{a-b +\omega c}{a} \leftarrow \kappa  , \quad \quad \frac{3c}{a}  \leftarrow \lambda , \quad \quad \frac{b-a}{a^{\alpha+1}}  \leftarrow \mu, \quad\quad   \frac{1}{a^{\alpha}} \leftarrow \nu, , \quad\quad  u \leftarrow v.
\end{equation*}
This yields the fractional KdV-BBM equation with general constants as 
\begin{equation*}
  u_t + \kappa u_x + \lambda uu_x - \mu D^{\alpha} u_x + \nu D^{\alpha} u_t=0,
\end{equation*}
 with  $\kappa, \mu \in \mathbb{R}$ and $\lambda, \nu \in \mathbb{R}^{+}$.

\subsection{Preliminaries}
In this subsection, we provide some basic facts for the standard KdV-BBM and the fractional KdV-BBM equation.
\subsubsection{Conservation Laws}
We assume that the solution $u(t,x)$ has either the compact support or decrease sufficiently fast at the infinity. Then, the following integrals are conserved by the solution of the equation  \eqref{fKdVBBM}:
\begin{eqnarray*}
\mathcal{I}_0(t)&=&\int_{\mathbb{R}} u(t,x)dx,\\
\mathcal{I}_1(t)&=&\int_{\mathbb{R}} \left( u(t,x)+ \nu D^{\alpha} u(t,x) \right)dx,\\
\mathcal{I}_2(t)&=&\int_{\mathbb{R}} \left( u^2(t,x)+ \nu |D^{\alpha/2}u(t,x)|^2 \right)dx.
\end{eqnarray*}
Besides the theoretical importance the invariants have practical one. For instance they may help to check the global accuracy of the proposed numerical scheme. 
\begin{remark}
It is worth noting here that the quantities $\mathcal{I}_0(t)$, $\mathcal{I}_1(t)$ and $\mathcal{I}_2(t)$ are conserved by the flow due to stability by Fourier transform for any functions in Schwartz space. 
\end{remark}
\subsubsection{Solitary Waves}
Solitary waves are known as localized traveling wave solutions of the form $u(t,x)=Q_c(\xi)$, $\xi=x-ct$ with constant wave speed $c$ and given condition $\lim_{|\xi| \rightarrow \infty} Q_c(\xi)=0.$
The standard KdV-BBM equation ($\alpha=2$) admits exact solitary wave solutions of the form 
$$ u(t,x)= 3\frac{c-\kappa}{\lambda}\operatorname{sech}^2\left( \frac{1}{2}   \sqrt{\frac{c-\kappa}{\nu c+ \mu}}(x-ct)   \right),$$
with given wave speed $c$, \cite{mitsotakis}.
Moreover, the solitary wave solution of the fractional KdV-BBM equation for $\alpha=1$ can be implicitly obtained as
$$ u(t,x)= \frac{4(c-1)}{1+\left[\frac{(c-1)}{c+1}\right]^2(x-ct)^2},$$
where $\kappa= \lambda= \mu =\nu =1$, \cite{oruc}. 

\setcounter{equation}{0}
\section{Existence in Large Time Scale}
\setcounter{equation}{0}
In order to remove the quadratic term in equation  \eqref{fKdVBBM}, one may use the standard approach normal form transformation \cite{shatah}. However, this approach is not applicaple for problems of this type because of the unboundedness of the transformation \cite{hunter}. Therefore, we will use the normal form transformation to obtain the modified energy functional for the Cauchy problem \eqref{fKdVBBM}-\eqref{u0}.  

\subsection{The Normal Form Transform}
In this subsection we aim to remove quadratically nonlinear term in equation \eqref{fKdVBBM}, by transforming it to the cubic nonlinear equation. To this end, we define a new variable $w$ and use the normal form transformation $w=u + P(u,u)$. Here, the bilinear form $P(u,u)$ is the pseudo-product and defined by the help of the Fourier transform, i.e.,
\begin{equation}\label{fourier}
  \mathcal{F}(P(q_1,q_2))(k) = \int_{\mathbb{R}} m(k- l,l) \hat{q_1}(k-l)\hat{q
_2}(l)d l,
\end{equation}
where $m$ can be determined following the further steps. First, we substitute the new variable $w$ and obtain the following PDE
\begin{align} \label{transform}
w_t+\kappa w_x  -\mu D^{\alpha}w_x +\nu D^{\alpha}w_t = & -\lambda uu_x  -(1+ \nu D^{\alpha})P((1+\nu D^{\alpha})^{-1} \kappa u_x,u)  \nonumber \\ &-(1+ \nu D^{\alpha})P(u,(1+\nu D^{\alpha})^{-1} \kappa u_x)  \nonumber \\ 
&+  (1+ \nu D^{\alpha})P((1+\nu D^{\alpha})^{-1} \mu D^{\alpha} u_x, u) \nonumber\\
& + (1+ \nu D^{\alpha})P(u,(1+\nu D^{\alpha})^{-1} \mu D^{\alpha} u_x)  \nonumber \\ 
 &+ (\kappa-\mu D^{\alpha})P(u,u)_x + R(u),
\end{align}
where
\begin{equation*}
  R(u)=-(1+\nu D^{\alpha})[P((1+\nu D^{\alpha})^{-1} \lambda u u_x, u) + P(u, (1+\nu D^{\alpha})^{-1} \lambda u u_x)].
\end{equation*}
Second, we write the transformed equation \eqref{transform}  in Fourier space, i.e.,
\begin{align} \label{ftransform}
\mathcal{F} (w_t+\kappa w_x  -\mu D^{\alpha}w_x  +\nu D^{\alpha}w_t) (k) =& \mathcal{F} \{ -\lambda uu_x  \nonumber\\
&-(1+ \nu D^{\alpha})P((1+\nu D^{\alpha})^{-1} \kappa u_x,u)  \nonumber \\ 
&-(1+ \nu D^{\alpha})P(u,(1+\nu D^{\alpha})^{-1} \kappa u_x)  \nonumber \\ 
&+  (1+ \nu D^{\alpha})P((1+\nu D^{\alpha})^{-1} \mu D^{\alpha} u_x, u)  \nonumber \\ 
& + (1+ \nu D^{\alpha})P(u,(1+\nu D^{\alpha})^{-1} \mu D^{\alpha} u_x) \nonumber \\ 
&+ (\kappa-\mu D^{\alpha})P(u,u)_x + R(u) \}(k).
\end{align}
First, we write the Fourier transform of the first term in right hand side of the above equation as
\begin{equation}\label{fourier0}
  \mathcal{F} \{ \lambda u u_x \}(k) =\frac{i}{2} \int_{\mathbb{R}} \lambda k \hat{u}(k-l) \hat{u}(l) dl.
\end{equation}
Moreover, we use the definition \eqref{fourier} for the other terms in the equation \eqref{ftransform}. These terms can be rewritten as
\begin{align} \label{fourier1}
  \mathcal{F} \{ (1+ \nu D^{\alpha})P((1+&\nu D^{\alpha})^{-1} \kappa u_x,u) \}(k) \nonumber \\ 
  &=i  \kappa \int_{\mathbb{R}} (1+ \nu |k |^{\alpha})\frac{ (k-l) m(k-l,l) }{1+ \nu |k-l|^{\alpha}}  \hat{u}(k-l) \hat{u}(l) dl, \nonumber \\
  \mathcal{F} \{ (1+ \nu D^{\alpha})P(u,(&1+\nu D^{\alpha})^{-1} \kappa u_x) \}(k) \nonumber \\
  & = i \kappa \int_{\mathbb{R}} (1+ \nu |k |^{\alpha})  \frac{l \  m(k-l,l)}{1+ \nu|l|^{\alpha}}  \hat{u}(k-l) \hat{u}(l) dl, \nonumber \\
  \mathcal{F} \{ (1+ \nu D^{\alpha})P((1+&\nu D^{\alpha})^{-1} \mu D^{\alpha} u_x,u) \}(k) \nonumber \\
 &= i \mu \int_{\mathbb{R}} (1+ \nu |k |^{\alpha})  \frac{ |k-l |^{\alpha + 1} m(k-l,l)}{1+ \nu|k-l|^{\alpha}}  \hat{u}(k-l) \hat{u}(l) dl, \nonumber \\
  \mathcal{F} \{ (1+ \nu D^{\alpha})P(u,(&1+\nu D^{\alpha})^{-1} \mu D^{\alpha} u_x) \}(k) \nonumber \\
&=i \mu \int_{\mathbb{R}} (1+ \nu |k |^{\alpha}) \frac{ |l |^{\alpha + 1}m(k-l,l)}{1+ \nu | l | ^{\alpha}}   \hat{u}(k-l) \hat{u}(l) dl, \nonumber \\
  \mathcal{F} \{  (\kappa-\mu D^{\alpha})P(u,u&)_x   \}(k) \nonumber \\
&= i \int_{\mathbb{R}} (\kappa-\mu |k |^{\alpha}) \ k \ m(k-l,l) \hat{u}(k-l) \hat{u}(l) dl.
  \end{align}
Let us use the transforms \eqref{fourier0} and \eqref{fourier1} to remove the quadratic terms in equation \eqref{ftransform}. Then, $m$ is expected to satisfy the following integral equation
\begin{align*}
  \int_{\mathbb{R}} &\Big (-\frac{1}{2}\lambda k +  m(k-l,l) \big[ (\kappa-\mu |k|^{\alpha})k - (1+\nu |k|^{\alpha}) (1+\nu |k-l|^{\alpha})^{-1} \kappa \ (k-l)   \nonumber \\
  & -  (1+\nu |k|^{\alpha}) (1+\nu |l|^{\alpha})^{-1} \kappa\ l +(1+\nu |k|^{\alpha})(1+\nu |k-l|^{\alpha})^{-1} \mu |k-l|^{\alpha}(k-l)          \nonumber \\
  & \hspace{104pt}+  (1+\nu |k|^{\alpha}) (1+\nu |l|^{\alpha})^{-1} \mu  |l|^{\alpha}l ] \Big ) \hat{u}(k-l) \hat{u}(l) dl= 0.
\end{align*}
According to calculations above $m$ is determined by
\begin{align*} \label{m}
m(k&-l,l) \\
&=\frac{ \lambda k (1+\nu |k-l|^{\alpha})(1+\nu |l|^{\alpha}) }{ 2(\kappa \nu + \mu) [k (1+\nu |l|^{\alpha})(|k-l|^{\alpha}-|k|^{\alpha}) + l (1+\nu |k|^{\alpha})(|l|^{\alpha}-|k-l|^{\alpha})]},
\end{align*}
which implies $m(k-l,l)=m(l,k-l)$. Hence, we conclude that $m$ is symmetric in  $k-l$ and $l$.
Under such a choice of $m$, \eqref{transform} finally turns into 
\begin{equation*}
w_t+\kappa w_x  -\mu D^{\alpha}w_x +\nu D^{\alpha}w_t= R(u).
\end{equation*}
\begin{proposition} Let $\alpha \in (0,1)$. Then, there exist $\mathcal{\underline{C}}_{\kappa,\lambda,\nu,\mu}$, $\mathcal{\bar{C}}_{\kappa,\lambda,\nu,\mu} >0$, such that
\begin{eqnarray}
 && \hspace{-30pt}\mathcal{\underline{C}}_{\kappa,\lambda,\nu,\mu}  \left(    \frac{|k-l|^{\alpha}(k-l)}{|l|}  + \frac{|l|^{\alpha}l}{|k-l|}   \right) \lesssim |m(k-l,l)| \nonumber \\
& &\hspace{80pt}\lesssim \mathcal{\bar{C}}_{\kappa,\lambda,\nu,\mu} \left(    \frac{|k-l|^{1-\alpha}}{|l|}  + \frac{|l|^{1-\alpha}}{|k-l|}   \right), \quad  (k-l)^2+l^2 \leq 1,\nonumber \\
\label{m_in1} \\
&&\hspace{-30pt} \mathcal{\underline{C}}_{\kappa,\lambda,\nu,\mu} \left(    \frac{|k-l|^{1-\alpha}}{|l|}  + \frac{|l|^{1-\alpha}}{|k-l|}   \right) \lesssim |m(k-l,l)|\nonumber\\ 
&& \hspace{62pt}\lesssim \mathcal{\bar{C}}_{\kappa,\lambda,\nu,\mu} \left(    \frac{|k-l|^{\alpha}(k-l)}{|l|}  + \frac{|l|^{\alpha}l}{|k-l|}   \right),  \quad (k-l)^2+l^2 \geq 1.  \nonumber \\
\label{m_in2}. 
\end{eqnarray}
\end{proposition}
\begin{proof}
First we propose polar coordinates: $k-l=r \cos(\theta)$ and $l= r\sin(\theta)$ with $r>0$ and $0 \leq \theta < 2\pi.$ Then, $m(k-l,l)$ becomes
\begin{equation}\label{m_polar}
m(r,\theta)= \frac{\lambda r (\cos(\theta) + \sin(\theta)) (1+ \nu |r \cos(\theta)|^{\alpha}) (1+ \nu |r \sin(\theta)|^{\alpha})}{2(\kappa \nu + \mu) M(r\cos(\theta), r\sin(\theta))},
\end{equation}
where
\begin{align*}
M(r&\cos(\theta),r\sin(\theta))  \\ &\hspace{-2pt}= r (\cos(\theta)+ \sin(\theta)) (1+ \nu |r \sin(\theta)|^{\alpha}) ( |r \cos(\theta)|^{\alpha}  - |r \sin(\theta)+ r\cos(\theta)|^{\alpha}) \\
&\hspace{54pt} + r \sin(\theta) (1+ \nu |r \sin(\theta)  + r\cos(\theta)|^{\alpha}) ( |r \sin(\theta)|^{\alpha} - |r\cos(\theta)|^{\alpha})\\ 
& =: r^{1+\alpha} \tilde{M}(r\cos(\theta), r\sin(\theta)).
\end{align*}
Note, that $\tilde{M}(r\cos(\theta), r\sin(\theta))=0$, whenever $\cos(\theta)=0$, $\sin(\theta)=0$ or $\cos(\theta)+\sin(\theta)=0.$ Therefore, $\tilde{M}(r\cos(\theta), r\sin(\theta))$ can be written in the following form
$$ \tilde{M}(r\cos(\theta), r\sin(\theta))=\cos(\theta)\sin(\theta)(\cos(\theta)+\sin(\theta))H(r,\theta).$$
Here, the function $H(r,\theta)$ is a bounded function and is also bounded away from zero and infinity for $r > 0$, which was shown in \cite{ehrns}. Now, equation \eqref{m_polar} turns
\begin{equation} \label{m_last}
m(r,\theta)= \frac{\lambda r^{2-\alpha}  (1+ \nu |r \cos(\theta)|^{\alpha}) (1+ \nu |r \sin(\theta)|^{\alpha})}{2(\kappa \nu + \mu)r\cos(\theta)r\sin(\theta) H(r,\theta)}.
\end{equation}
Since $\kappa,\lambda,\nu,\mu$ are postive constants then we conclude that the term $ \frac{\lambda}{2(\kappa \nu + \mu)H(r,\theta)}$ in \eqref{m_last} is also a bounded function such that
$$
\mathcal{\underline{C}}_{\kappa,\lambda,\nu,\mu} \leq \frac{\lambda}{2(\kappa \nu + \mu)H(r,\theta) } \leq \mathcal{\bar{C}}_{\kappa,\lambda,\nu,\mu},
$$
which completes the proof.
\end{proof}


\subsection{The Modified Energy}
The vanishing viscosity method guarantees the existence and uniqueness of solution $u \in {C}([0,T],H^{N+\frac{\alpha}{2}}(\mathbb{R}) )$ to problem \eqref{fKdVBBM}-\eqref{u0} with  $ \frac{1}{\epsilon} \lesssim  T $ and $N\geq 2$. The proof follows taking into account the general constants $\kappa, \lambda, \mu, \nu $ instead of $1,1,-\frac{3}{4}, \frac{5}{4}$ in the proof of Theorem 2.7 \cite{oruc}.
Since it seems difficult to obtain the global well-posedness  when $ 0 < \alpha < 1$, we aim to prove Theorem \ref{theorem} with an extention of the existence time. To this end, we need an a priori $H^{N+\frac{\alpha}{2}}$-bound for $u \in {C}([0,T],H^{N+\frac{\alpha}{2}}(\mathbb{R}))$. Here we use the technique given in \cite{hunter} and propose the {modified energy} for the original equation \eqref{fKdVBBM}, because the standard energy estimate can not be applied to an equation of the form
\begin{equation} \label{transform2}
w_t+\kappa w_x +\nu D^{\alpha}w_t -\mu D^{\alpha}w_x = R(u),
\end{equation}
where $R(u)$ has a cubic type nonlinearity. By the help of normal transform $w=u + P(u,u)$ proposed in previous subsection, we can write
\begin{align} \label{parenergy}
 \| \partial_x^n(1+\nu D^{\alpha})^{1/2} w \|_{L^2}^2 &=  \| \partial_x^n(1+\nu D^{\alpha})^{1/2} u \|_{L^2}^2 \nonumber \\ 
&+  2 \langle \partial_x^n(1+\nu D^{\alpha})^{{1}/{2}}  u, \partial_x^n(1+\nu D^{\alpha})^{{1}/{2}} P(u,u) \rangle \nonumber\\
&  +\| \partial_x^n(1+\nu D^{\alpha})^{{1}/{2}} P(u,u)\|_{L^2}^2.
\end{align}
Neglecting the higher derivative quartic term in \eqref{parenergy}, we define the $n^{th}$ partial energy as follows
\begin{equation}\label{nparenergy}
E_n(t)= \| \partial_x^n(1+\nu D^{\alpha})^{1/2} u \|_{L^2}^2 +  2 \langle \partial_x^n(1+\nu D^{\alpha})^{1/2} u, \partial_x^n(1+\nu D^{\alpha})^{1/2} P(u,u) \rangle.
\end{equation}
The modified energy is defined by taking the sum over $n$ from $1$ to $N$ and simply adding $L_2$-term to the total energy, which corresponds to $n=0.$ Then, the modified energy is almost equivalent to the $H^{N+\alpha/2}(\mathbb{R})$ Sobolev energy for sufficiently small data.
\begin{lemma} Let $E_n$ be defined by \eqref{nparenergy} and $\alpha \in (0,1)$. Then, there exists $\epsilon>0$ such that
\begin{equation}\label{En}
E^{(N)}(t) :=\sum_{n=1}^{N}E_n(t)+ \|   (1+\nu D^{\alpha})^{1/2} u    \|^2_{L^2(\mathbb{R})} \simeq   \|  u \| ^2_{H^{N+\alpha/2}(\mathbb{R})},
\end{equation}
uniformly for   $\|  u \|_{H^{N+\alpha/2}(\mathbb{R})} < \epsilon.$
\end{lemma}
\begin{proof}
Let us rewrite equation \eqref{En} in the following form
\begin{align*}
 E^{(N)}(t)  =  \|   (1+\nu D^{\alpha})^{1/2} u    \|^2_{L^2(\mathbb{R})} + \sum_{n=1}^{N} \Big( \| \partial_x^n(1+\nu D^{\alpha})^{1/2} u \|_{L^2}^2 \\
 + 2   \langle \partial_x^n(1+\nu D^{\alpha})^{1/2} u, \partial_x^n(1+\nu D^{\alpha})^{1/2} P(u,u) \rangle \Big). \nonumber 
\end{align*}
The first two terms in above equation can be controlled by the  $H^{N+\alpha/2}(\mathbb{R})$-Sobolev norm whenever $\|u\|_{H^{N+\alpha/2}(\mathbb{R})}<\epsilon.$ Therefore, we focus on the last one $ \langle \partial_x^n(1+\nu D^{\alpha})^{1/2} u, \partial_x^n(1+\nu D^{\alpha})^{1/2} P(u,u) \rangle.$
In the first step we use the symmetry of $P$ and write the following decomposition.
\begin{align*}
 &2  \langle (1+\nu D^{\alpha})^{1/2}  \partial_x^n u, (1+\nu D^{\alpha})^{1/2} P(u, \partial_x^n  u) \rangle \\ &+ \sum_{j=1}^{n-1} b_{n,j}  \langle (1+\nu D^{\alpha})^{1/2} \partial_x^n u, (1+\nu D^{\alpha})^{1/2} P(  \partial_x^j u, \partial_x^{n-j} u) \rangle 
:=2A_0 + \sum_{j=1}^{n-1} b_{n,j} A_j.
 \end{align*}
Here, we denote the binomial coefficients in the expression by $b_{n,j}$. As stated before, the challenging term is $A_0$ beacuse of the properties $m.$ By integration by parts and a change of variables, we have
\begin{align*}
A_0 &= \int_{\mathbb{R}}\int_{\mathbb{R}} m(k-l,l) \hat{u}(k-l)(il)^{n}  \hat{u}(l)  (1+\nu|k|^{\alpha}) \overline{ (ik)^{n} \hat{u}(k)}   dl dk    \\
&=  \int_{\mathbb{R}}\int_{\mathbb{R}} m(k-l,l) \hat{u}(k-l)i (k-l)  (il)^{n}  \hat{u}(l)  (1+\nu|k|^{\alpha}) \overline{ (ik)^{n-1} \hat{u}(k)}   dl dk     \\
& \hspace{15pt}- \int_{\mathbb{R}}\int_{\mathbb{R}} m(k-l,l) \hat{u}(k-l)(il)^{n+1}  \hat{u}(l)  (1+\nu|k|^{\alpha}) \overline{ (ik)^{n-1} \hat{u}(k)}   dl dk  \\
&=:A_0^1 + A_0^2.
\end{align*}
In a first step, we estimate $A_0^1$. Notice that
\begin{align*}
&m(k-l,l)\\
&= \frac{\lambda k (1+\nu |k-l|^{\alpha})(1+\nu |l|^{\alpha}) }{ 2(\kappa \nu + \mu) [k (1+\nu |l|^{\alpha})(|k-l|^{\alpha}-|k|^{\alpha}) + l (1+\nu |k|^{\alpha})(|l|^{\alpha}-|k-l|^{\alpha})]},
\end{align*}
has singularities at $k=0$, $l=0$ and $k-l=0$. On the other hand, low frequency singularities can be cancelled by the term $(k-l)(il)^n \overline{(ik)^{n-1}}$ appearing in $A_0^1.$ Thus, we first rewrite  $A_0^1$ as 
\begin{align}\label{aa01}
A_0^1 =  \int \int_{|k-l|^2+|l|^2 \leq 1} m(k-l,l) \hat{u}(k-l)i (k-l)  (il)^{n}  \hat{u}(l)  (1+\nu|k|^{\alpha}) \nonumber \\ 
\times \overline{ (ik)^{n-1} \hat{u}(k)}   dl dk  \nonumber \\
+ \int \int_{|k-l|^2+|l|^2 \geq 1} m(k-l,l) \hat{u}(k-l)i (k-l)  (il)^{n}  \hat{u}(l)  (1+\nu|k|^{\alpha}) \nonumber \\
\times \overline{ (ik)^{n-1} \hat{u}(k)}   dl dk, 
\end{align}
and then we consider high frequencies by using  \eqref{m_in2}
\begin{align*}
 |m(k-l,l)| &\lesssim \mathcal{\bar{C}}_{\kappa,\lambda,\nu,\mu} \left(    \frac{|k-l|^{1+\alpha}}{|l|}  + \frac{|l|^{1+\alpha}}{|k-l|}   \right)\\ &= \mathcal{\bar{C}}_{\kappa,\lambda,\nu,\mu} \left(    \frac{|k-l|^{1+\alpha}}{|l|}  + \frac{|k-l-k|^{1+\alpha}}{|k-l|}   \right)\\& \leq \mathcal{\bar{C}}_{\kappa,\lambda,\nu,\mu} \left(    \frac{|k-l|^{1+\alpha}}{|l|}  +|k-l|^{\alpha}+ \frac{|k|^{1+\alpha}}{|k-l|}   \right), \quad |k-l|^2+|l|^2 \geq 1.
\end{align*}
Now, we use this inequality to estimate the second integral in \eqref{aa01}. Let us take the modulus of the term as
\begin{align*}
\Big | \int & \int_{|k-l|^2+|l|^2 \geq 1} m(k-l,l) \hat{u}(k-l)i (k-l)  (il)^{n}  \hat{u}(l)  (1+\nu|k|^{\alpha}) \overline{ (ik)^{n-1} \hat{u}(k)}   dl dk \Big|  \\
&\lesssim \int_{\mathbb{R}}\int_{\mathbb{R}} |\mathcal{\bar{C}}_{\kappa,\lambda,\nu,\mu} (k-l)^{2+\alpha}     \hat{u}(k-l)  l^{n-1}  \hat{u}(l)  (1+\nu|k|^{\alpha}) k^{n-1} \hat{u}(k)| dl dk \\
&\hspace{10pt} + \int_{\mathbb{R}}  \int_{\mathbb{R}} | \mathcal{\bar{C}}_{\kappa,\lambda,\nu,\mu} (k-l)^{1+\alpha}  \hat{u}(k-l)  l^{n}     \hat{u}(l)  (1+\nu|k|^{\alpha}) k^{n-1} \hat{u}(k) | dl dk \\
& \hspace{10pt} +  \int_{\mathbb{R}} \int_{\mathbb{R}} |\mathcal{\bar{C}}_{\kappa,\lambda,\nu,\mu} k^{n+\alpha}       \hat{u}(k-l)  l^{n}    \hat{u}(l)   (1+\nu|k|^{\alpha}) {\hat{u}(k)}|   dl dk\\
&:=A_0^{1,1} +A_0^{1,2} +A_0^{1,3}.
\end{align*}
Applying Cauchy-Schwarz inequality and Sobolev embeddings, one can estimate these terms as follows
\begin{align}\label{a01}
A_0^{1,1}  &= \int_{\mathbb{R}}\int_{\mathbb{R}} |\mathcal{\bar{C}}_{\kappa,\lambda,\nu,\mu} (k-l)^{2+\alpha}     \hat{u}(k-l)  l^{n-1}  \hat{u}(l)  (1+\nu|k|^{\alpha}) k^{n-1} \hat{u}(k)| dl dk        \nonumber \\
	     &\leq  \mathcal{\bar{C}}_{\kappa,\lambda,\nu,\mu} \|  \partial_{xx}D^{\alpha}u \partial_x^{n-1} u \|_{L^2(\mathbb{R})}   \| ( 1+\nu D^{\alpha} )  \partial_x^{n-1}u \|_{L^2(\mathbb{R})}    \nonumber \\
	     &\leq   \mathcal{\bar{C}}_{\kappa,\lambda,\nu,\mu} \| u \|_{H^{2}(\mathbb{R})} \| D^{\alpha}\partial_x^{n-1} u \|_{L^{\infty}(\mathbb{R})}   \| u \|_{H^{n-1+\alpha}(\mathbb{R})}     \nonumber \\
	     &\lesssim \mathcal{\bar{C}}_{\kappa,\lambda,\nu,\mu} \| u \|_{H^{2}(\mathbb{R})}  \| u \|^2_{H^{n}(\mathbb{R})},
\end{align}

\begin{eqnarray}\label{a012}
A_0^{1,2}  &=& \int_{\mathbb{R}}\int_{\mathbb{R}} |\mathcal{\bar{C}}_{\kappa,\lambda,\nu,\mu} (k-l)^{1+\alpha}     \hat{u}(k-l)  l^{n}  \hat{u}(l)  (1+\nu|k|^{\alpha}) k^{n-1} \hat{u}(k)| dl dk        \nonumber \\
	     &\leq&  \mathcal{\bar{C}}_{\kappa,\lambda,\nu,\mu} \|  \partial_{x}D^{\alpha}u ( 1+\nu D^{\alpha} )  \partial_x^{n-1}u       \|_{L^2(\mathbb{R})}   \|   \partial_x^{n} u    \|_{L^2(\mathbb{R})}    \nonumber \\
	     &\leq&   \mathcal{\bar{C}}_{\kappa,\lambda,\nu,\mu}  \|  \partial_{x}D^{\alpha}u \|_{L^{\infty}(\mathbb{R})} \| u \|_{H^{n-1+\alpha}(\mathbb{R})} \|  u \|_{H^{n}(\mathbb{R})} \nonumber \\
	     &\lesssim& \mathcal{\bar{C}}_{\kappa,\lambda,\nu,\mu} \| u \|_{H^{2}(\mathbb{R})}  \| u \|^2_{H^{n}(\mathbb{R})}
\end{eqnarray}
 and
\begin{align} \label{a013}
A_0^{1,3}  &= \int_{\mathbb{R}}\int_{\mathbb{R}} |\mathcal{\bar{C}}_{\kappa,\lambda,\nu,\mu} k^{n+\alpha}  \hat{u}(k-l)  l^{n}  \hat{u}(l)  (1+\nu|k|^{\alpha}) \hat{u}(k)| dl dk        \nonumber \\
	      &=  \mathcal{\bar{C}}_{\kappa,\lambda,\nu,\mu}  \int_{\mathbb{R}}\int_{\mathbb{R}} |  (k-l+l)^{1+\alpha}\hat{u}(k-l)  l^{n}  \hat{u}(l)  (1+\nu|k|^{\alpha})k^{n-1} \hat{u}(k)| dl dk \nonumber \\
	     &\leq \mathcal{\bar{C}}_{\kappa,\lambda,\nu,\mu}  \int_{\mathbb{R}}\int_{\mathbb{R}} |  |k-l|^{1+\alpha} \hat{u}(k-l)  l^{n}  \hat{u}(l)  (1+\nu|k|^{\alpha})k^{n-1} \hat{u}(k)| dl dk   \nonumber \\
	     &+\mathcal{\bar{C}}_{\kappa,\lambda,\nu,\mu}  \int_{\mathbb{R}} \int_{\mathbb{R}}   |\hat{u}(k-l)  |l|^{n+1+\alpha}  \hat{u}(l)  (1+\nu|k|^{\alpha})k^{n-1} \hat{u}(k)| dl dk    \nonumber \\
	     &\leq  \mathcal{\bar{C}}_{\kappa,\lambda,\nu,\mu}  \Big(  \|  \partial_{x}D^{\alpha}u \partial_x^{n} u \|_{L^2(\mathbb{R})}   \| ( 1+\nu D^{\alpha} )  \partial_x^{n-1}u \|_{L^2(\mathbb{R})} \nonumber \\ 
               & \hspace{140pt} +  \| u \partial_{x}D^{\alpha} \partial_x^{n} u \|_{L^2(\mathbb{R})}   \| ( 1+\nu D^{\alpha} )  \partial_x^{n-1}u \|_{L^2(\mathbb{R})}  \Big)  \nonumber \\
	    &\leq \mathcal{\bar{C}}_{\kappa,\lambda,\nu,\mu} \Big( \| u \|_{H^{1+\alpha}(\mathbb{R})}  \| u \|_{H^{n}(\mathbb{R})}  \| u \|_{H^{n-1+\alpha}(\mathbb{R})} \nonumber \\  &\hspace{140pt}  +    \| u \|_{H^{1+\alpha}(\mathbb{R})}  \| u \|_{H^{n}(\mathbb{R})} \| u \|_{H^{n-1+\alpha}(\mathbb{R})}   \Big) \nonumber \\
	    &\lesssim \mathcal{\bar{C}}_{\kappa,\lambda,\nu,\mu} \| u \|_{H^{2}(\mathbb{R})}  \| u \|^2_{H^{n}(\mathbb{R})}.
\end{align}
Thus, the inequalties \eqref{a01}, \eqref{a012} and \eqref{a013}  yield that
\begin{equation}\label{a1}
|A_0^1| \lesssim \mathcal{\bar{C}}_{\kappa,\lambda,\nu,\mu}  \| u \|_{H^{2}(\mathbb{R})}  \| u \|^2_{H^{n}(\mathbb{R})}.
\end{equation}
The second step is devoted to the estimation of $A_0^2.$
$$
A_0^2=- \int_{\mathbb{R}}\int_{\mathbb{R}} m(k-l,l) \hat{u}(k-l)(il)^{n+1}  \hat{u}(l)  (1+\nu|k|^{\alpha}) \overline{ (ik)^{n-1} \hat{u}(k)}   dl dk$$
First, we replace $-(k,l)$ by $(k,l)$
$$
A_0^2=- \int_{\mathbb{R}}\int_{\mathbb{R}} m(k-l,l) \hat{u}(l-k)(ik)^{n-1}  \hat{u}(k)  (1+\nu|k|^{\alpha}) \overline{ (il)^{n+1} \hat{u}(l)}   dl dk.$$
Then, we use the change of variable from $(k,l)$ to $(l,k)$ and the relation $$ m(k-l,l)l (1+\nu|k|^{\alpha}) + m(l-k,k)k(1+ \nu |l|^{\alpha})=0,$$
respectively, so that
\begin{align} \label{a2a1}
A_0^2&=- \int_{\mathbb{R}}\int_{\mathbb{R}} m(l-k,k) \hat{u}(k-l)(il)^{n-1}  \hat{u}(l)  (1+\nu|l|^{\alpha}) \overline{ (ik)^{n+1} \hat{u}(k)}   dl dk, \nonumber \\
&= -\int_{\mathbb{R}}\int_{\mathbb{R}} m(k-l,l) \hat{u}(k-l)(il)^{n}  \hat{u}(l)  (1+\nu|k|^{\alpha}) \overline{ (ik)^{n} \hat{u}(k)}   dl dk, \nonumber \\
&= -A_0= A_0^1-A_0^2,
\end{align}
which implies $2 A_0=A_0^1$.
Now, it remains to estimate the terms $A_j$ that can be given by
\begin{align*}
A_j= \int_{\mathbb{R}}\int_{\mathbb{R}} m(k-l,l) [i(k-l)]^j\hat{u}(k-l)(il)^{(n-j)}  \hat{u}(l)  (1+&\nu|k|^{\alpha}) \overline{ (ik)^{n} \hat{u}(k)}  dl dk, \\
  &\mbox{for} \quad  j=1,...,n-1.
\end{align*}
We follow the similar approach as for $A_0.$ Here, we focus on the high frequencies singularities in $m$ and use the  \eqref{m_in2}.
\begin{align*}
\Big| \int  & \int_{|k-l|^2+|l|^2 \geq 1}  \hspace{-20pt}  m(k-l,l) [i(k-l)]^j\hat{u}(k-l)(il)^{(n-j)}  \hat{u}(l)  (1+\nu|k|^{\alpha}) \overline{ (ik)^{n} \hat{u}(k)}   dl dk    \Big| \\
&\hspace{10pt} \lesssim \mathcal{\bar{C}}_{\kappa,\lambda,\nu,\mu}\int_{\mathbb{R}}\int_{\mathbb{R}} | (k-l)^{1+\alpha +j}     \hat{u}(k-l)  l^{n-1-j}  \hat{u}(l)  (1+\nu|k|^{\alpha}) k^{n} \hat{u}(k)| dl dk \\
& \hspace{25pt} +  \mathcal{\bar{C}}_{\kappa,\lambda,\nu,\mu}\int_{\mathbb{R}}  \int_{\mathbb{R}} | (k-l)^{j-1}  \hat{u}(k-l)  l^{n+1+\alpha-j}     \hat{u}(l)  (1+\nu|k|^{\alpha}) k^{n} \hat{u}(k) | dl dk \\
&\hspace{10pt}  :=A_j^{1} +A_j^{2}.
\end{align*}
We  can estimate these terms as in the following.
\begin{align}\label{aj1}
A_j^{1}  = \mathcal{\bar{C}}_{\kappa,\lambda,\nu,\mu}\iint_{\mathbb{R}^2} \hspace{-3pt}  | (k-l)^{1+\alpha +j}     \hat{u}(k-l)  l^{n-1-j}  \hat{u}(l) (1+\nu|k-l+l|^{\alpha/2}|k|^{\alpha/2})&\nonumber \\
 \times   k^{n} \hat{u}(k)| dl dk&      \nonumber \\
= \mathcal{\bar{C}}_{\kappa,\lambda,\nu,\mu}  \iint_{\mathbb{R}^2} \hspace{-3pt} | (k-l)^{1+\alpha +j} (1+\nu|k-l|^{\alpha/2})     \hat{u}(k-l)  l^{n-1-j}  \hat{u}(l)  (1+\nu|k|^{\alpha/2}) 
&\nonumber \\
 \times k^{n} \hat{u}(k)| dl dk     & \nonumber \\
 + \mathcal{\bar{C}}_{\kappa,\lambda,\nu,\mu}\iint_{\mathbb{R}^2} \hspace{-3pt}  | (k-l)^{1+\alpha +j}     \hat{u}(k-l)  l^{n-1-j} (1+\nu|l|^{\alpha/2})  \hat{u}(l)  (1+\nu|k|^{\alpha/2})
&\nonumber \\
 \times  k^{n} \hat{u}(k)| dl dk  &   \nonumber \\
\leq  \mathcal{\bar{C}}_{\kappa,\lambda,\nu,\mu} \|  D^{\alpha}\partial_x^{j+1}(1+\nu D^{\alpha/2})  u \partial_x^{n-1-j} u \|_{L^2}   \| ( 1+\nu D^{\alpha/2} )  \partial_x^{n}u \|_{L^2} \hspace{50pt}  &  \nonumber \\
+\mathcal{\bar{C}}_{\kappa,\lambda,\nu,\mu}  \|  D^{\alpha}\partial_x^{j+1}  u \partial_x^{n-1-j}   (1+\nu D^{\alpha/2})u \|_{L^2}   \| ( 1+\nu D^{\alpha/2} )  \partial_x^{n}u \|_{L^2} \hspace{40pt} & \nonumber \\
= \mathcal{\bar{C}}_{\kappa,\lambda,\nu,\mu} \|  \partial_x^{j+1}(1+\nu D^{\alpha/2})  u D^{\alpha}  \partial_x^{n-1-j} u \|_{L^2}   \| ( 1+\nu D^{\alpha/2} )  \partial_x^{n}u \|_{L^2} \hspace{50pt} & \nonumber\\
+  \mathcal{\bar{C}}_{\kappa,\lambda,\nu,\mu}   \|  \partial_x^{j+1}  u D^{\alpha} \partial_x^{n-1-j}  (1+\nu D^{\alpha/2})u \|_{L^2}   \| ( 1+\nu D^{\alpha/2} )  \partial_x^{n}u \|_{L^2} \hspace{40pt} & \nonumber \\
\lesssim   \mathcal{\bar{C}}_{\kappa,\lambda,\nu,\mu} \left( \| u \|_{H^{n-1}(\mathbb{R})} \|  u \|^2_{H^{n+\alpha/2}(\mathbb{R})}  + \| u \|_{H^{n-1}(\mathbb{R})} \|  u \|^2_{H^{n+\alpha/2}(\mathbb{R})} \right),\hspace{50pt}&
\end{align}
where we have used Plancherel's identity in the fourth step and in a similar manner 
\begin{equation}\label{aj2}
A_j^{2} \lesssim \mathcal{\bar{C}}_{\kappa,\lambda,\nu,\mu} \left( \| u \|_{H^{n}(\mathbb{R})} \|  u \|^2_{H^{n+\alpha/2}(\mathbb{R})}  + \| u \|_{H^{n-1}(\mathbb{R})} \|  u \|^2_{H^{n+\alpha/2}(\mathbb{R})} \right).
\end{equation}
By using \eqref{aj1} and \eqref{aj2} one can deduce that these terms are bounded by
\begin{equation}\label{aj}
|A_j| \lesssim  \mathcal{\bar{C}}_{\kappa,\lambda,\nu,\mu} \epsilon  \| u \|^2_{H^{n+\alpha/2}(\mathbb{R})},\quad \quad \quad 1 \leq j \leq n-1.
\end{equation}
Combining the inequalities \eqref{a1},   \eqref{a2a1} and \eqref{aj} for $1 \leq n \leq N$, we get \eqref{En} as long as $\|  u \|_{H^{N+\alpha/2}(\mathbb{R})} < \epsilon$. This completes the proof of lemma.
\end{proof}

\subsection{The Energy Estimates}
\begin{proposition}\label{proposition}
Let $n \geq 1$. Then
$$ \frac{d}{dt}E_n(t) \lesssim \|u\|^2_{H^{2+\alpha/2}(\mathbb{R})} \|u\|^2_{H^{n+\alpha/2}(\mathbb{R})} + \|u\|^4_{H^{n}(\mathbb{R})}.$$
\end{proposition}
\begin{proof}
We note that the proof is guaranteed by the proof of Proposition 4.1 in \cite{nilsson} taking into account the different constants.
\end{proof}
As the last step, we now propose the proof of our main theorem.
\begin{proof}[Proof of Theorem \ref{theorem}]
First we note that the $L_2-$term in the modified energy is conserved in time such that
$$ \| (1+\nu D^{\alpha}) ^{1/2}u(t) \|_{L^2(\mathbb{R})}^2 = \| (1+\nu D^{\alpha}) ^{1/2}u_0 \|_{L^2(\mathbb{R})}^2.$$
Then, summation over $n$ from $1$ to $N$ in the Proposition \ref{proposition} yields 
\begin{equation*}
\frac{d}{dt} E^{(N)}(t) \lesssim \| u(t) \|_{H^{N+\alpha/2}(\mathbb{R})}^4,
\end{equation*}
which implies 
\begin{align*}
\sum_{n=1}^{N} E_{n}(t) +\| (1+\nu D^{\alpha}) ^{1/2}u(t) \|_{L^2(\mathbb{R})}^2           \lesssim \sum_{n=1}^{N} E_{n}(0)   + \| (1+\nu D^{\alpha}) ^{1/2}u_0\|_{L^2(\mathbb{R})}^2  \\ +   \int_0^t   \| u(\bar{\tau},.) \|_{H^{N+\alpha/2}(\mathbb{R})}^4 d\bar{\tau}.
\end{align*}
Since the modified energy is almost equivalent to the $H^{N+\alpha/2}-$ energy such that 
$$E^{(N)}(t)  \simeq   \|  u \| ^2_{H^{N+\alpha/2}(\mathbb{R})},$$
then  we obtain 
\begin{equation*}
         \|  u(t) \| ^2_{H^{N+\alpha/2}(\mathbb{R})}  \lesssim \|  u_0 \| ^2_{H^{N+\alpha/2}(\mathbb{R})} +   \int_0^t   \| u(\bar{\tau},.) \|_{H^{N+\alpha/2}(\mathbb{R})}^4 d\bar{\tau}.
\end{equation*}
By using Grönwall's inequality, one can show that the local existence time of unique solution $u \in {C}([0,T]; H^{N+\alpha/2}(\mathbb{R}))$ can be extended until the time $T$, satisfies $ \frac{1}{\epsilon^2} \lesssim  T.$

\end{proof}

\section{Numerical Results}
In this section, we propose a pseudospectral method for the numerical integration of the fractional KdV-BBM equation \eqref{fKdVBBM}.
The spatial derivatives are computed spectrally  and the time integration carried out by using fourth order Runge-Kutta scheme. We refer the reader to \cite{oruc} for more details about the implementation of the method. 

Our main result, Theorem \ref{theorem}, guarantees that the maximal existence time satisfies $ \frac{1}{\epsilon^2} \lesssim  T.
$ On the other hand, existence of a solution beyond $ \frac{1}{\epsilon^2} $ remains as open question because the global theory is not available yet for the fractional type equations. Therefore, we carry out several numerical experiments in order to investigate $i)$ the qualitative behaviour of the solutions if it is global, $ii)$ blow-up issues,
$iii)$ dependence of the existence time on the initial data, and $iv)$ the long time behavior of small initial data.

In the numerical experiments we assume  $\kappa= \lambda= \mu =\nu =1 $ and consider the truncated domain $(x,t) \in [-L,L] \times [0,T_{max}]$ with periodic boundary conditions
$u(-L,t)=u(L,t)$. Inspired by the existence of an exact solution, we consider the initial guess in our experiments for $\alpha<1$ as follows
$$ u_0(x)=\delta \operatorname{sech}^2(x), $$
where $\delta$ is a real constant. 

In the first numerical experiment we consider the interval $\frac{1}{3}< \alpha < 1$  and study the time evolution of a large amplitude solution with initial guess $u_0=20\operatorname{sech}^2(x)$ and $\alpha=0.5$. We choose space interval $20[-\pi,\pi]$ with grid points $N=2^{14}$ and final time $T_{max}=5$ with time step $N_t=5 \times 10^{4}.$  In Figure \ref{solitary1}, we observe that the initial hump with large amplitude decomposes into two humps after a certain time. Later, these two humps continue to propagate without changing both their shapes and speeds, which is an indication of soliton type wave appearence. The oscillations propagate to the left unlike humps show the radiation of the remaining energy. This is similar to what has been found for the solutions to the fractional KdV equation when $\alpha > \frac{1}{2}$ and to the fractional BBM equation when  $\alpha > \frac{1}{3}$ in \cite{klein}. 

\begin{figure}[ht!]
 \centering
 \includegraphics[width=4.1in]{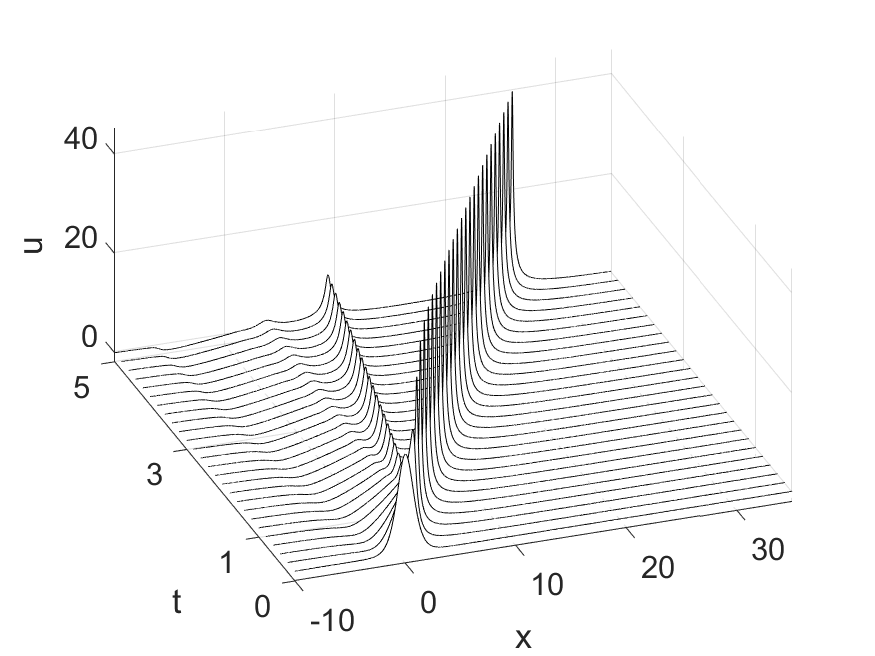}
  \caption{The propagation of solitary wave to the fractional KdV-BBM equation \eqref{fKdVBBM} for $\alpha=0.5$.}
 \label{solitary1}
\end{figure}

In Figure \ref{energy}, one can see the increase in the amplitude due to steepness of the solution and it becomes constant after some time. The same figure shows that the relative energy is conserved up to order of $10^{-6}$ and the solution is well resolved in Fourier space. So the numerical result is reliable and this case does not imply blow-up in the solution.
\begin{center}
\begin{figure}[ht!] 
\begin{minipage}[t]{0.35\linewidth}
   \includegraphics[width=2.5in]{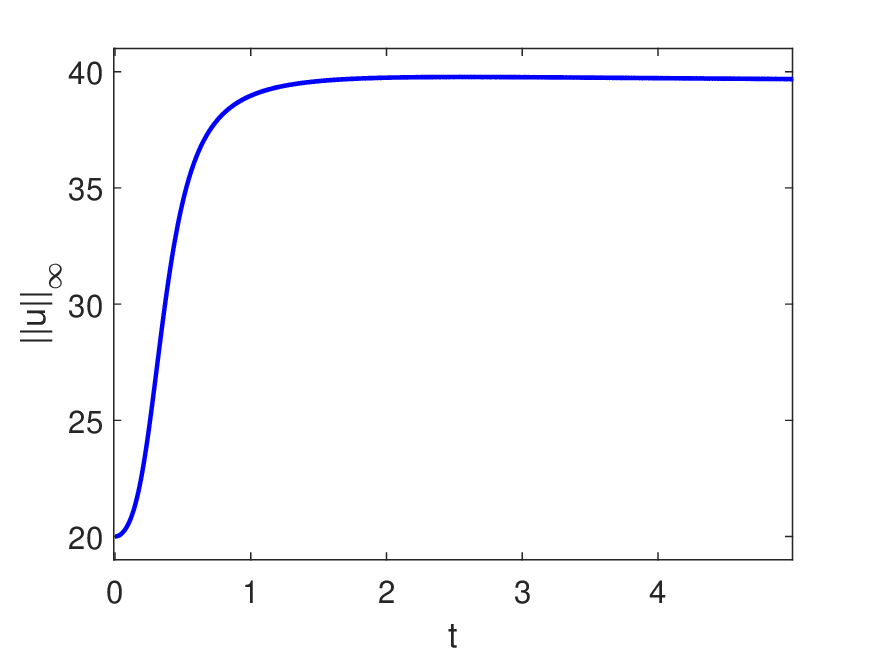}
 \end{minipage}
 \hspace{30pt}
 \begin{minipage}[t]{0.35\linewidth}
   \includegraphics[width=2.5in]{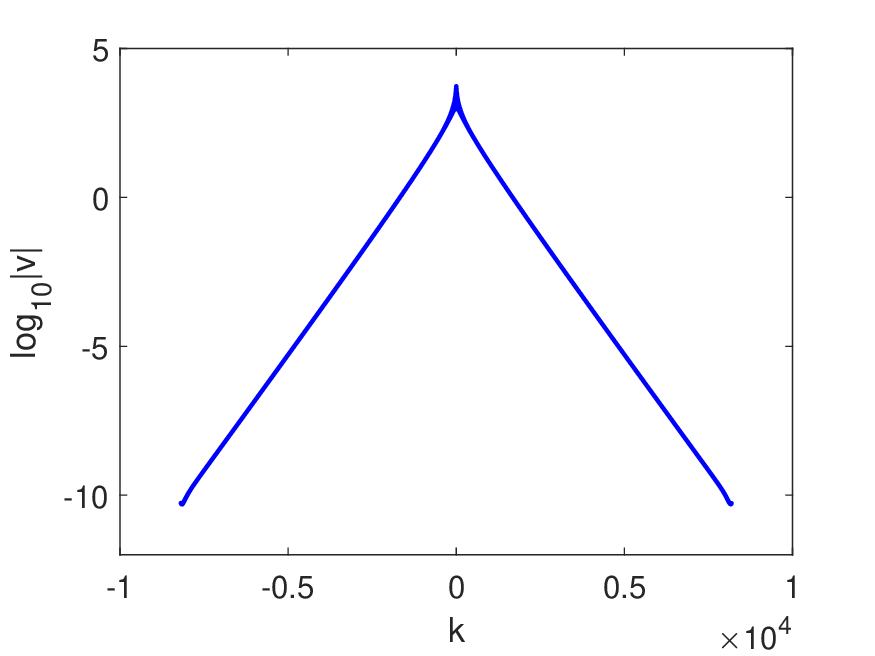}
 \end{minipage}
\center
 \hspace{-55pt} 
\begin{minipage}[t]{0.35\linewidth}
   \includegraphics[width=2.5in]{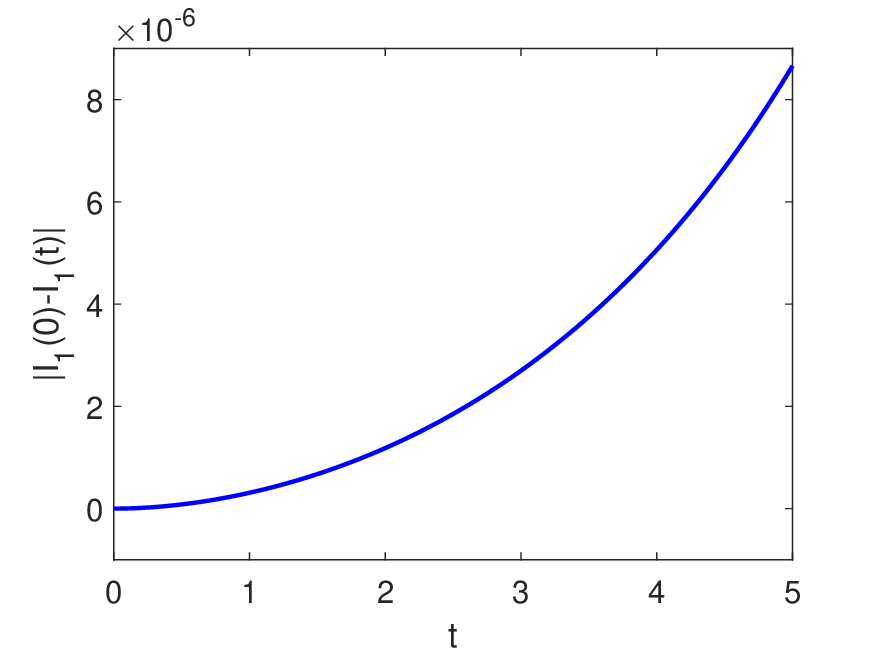}
 \end{minipage}
  \caption{ Variation in the $L_{\infty}$ norm of the solution to the fractional KdV-BBM equation \eqref{fKdVBBM} for $\alpha=0.5$ and $\delta=20$  (upper left panel),  modulus of the Fourier coefficients for the corresponding solution  (upper right panel) and variation of the energy identiy $I_1$ in time (bottom panel)}
  \label{energy}
\end{figure}
\end{center}

In the second numerical experiment, we consider $\alpha=0.9 > \frac{1}{3} $ and study a small amplitude solution with initial guess $u_0=0.1\operatorname{sech}^2(x)$. We choose the space interval $100[-\pi,\pi]$ with grid points $N=2^{15}$ and final time $T_{max}=50$ with time step $N_t=10^{5}.$  In Figure \ref{a09}, the numerical solution at the final time is depicted in the upper left panel and it is simply radiated away to infinity. In the right panel of the same figure,  $L_{\infty}$ norm of the numerical solution decreases with time. The bottom panel shows the well resolution of the solution in Fourier space. This experiment also does not indicate the existence of blow-up.

\begin{center}
\begin{figure}[ht!] 
\begin{minipage}[t]{0.35\linewidth}
   \includegraphics[width=2.5in]{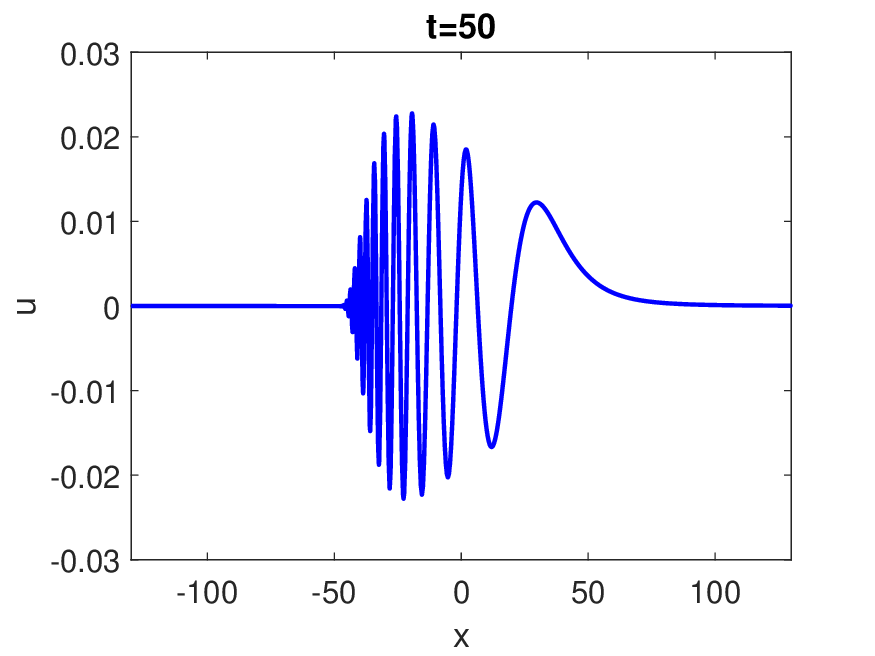}
 \end{minipage}
 \hspace{30pt}
 \begin{minipage}[t]{0.35\linewidth}
   \includegraphics[width=2.5in]{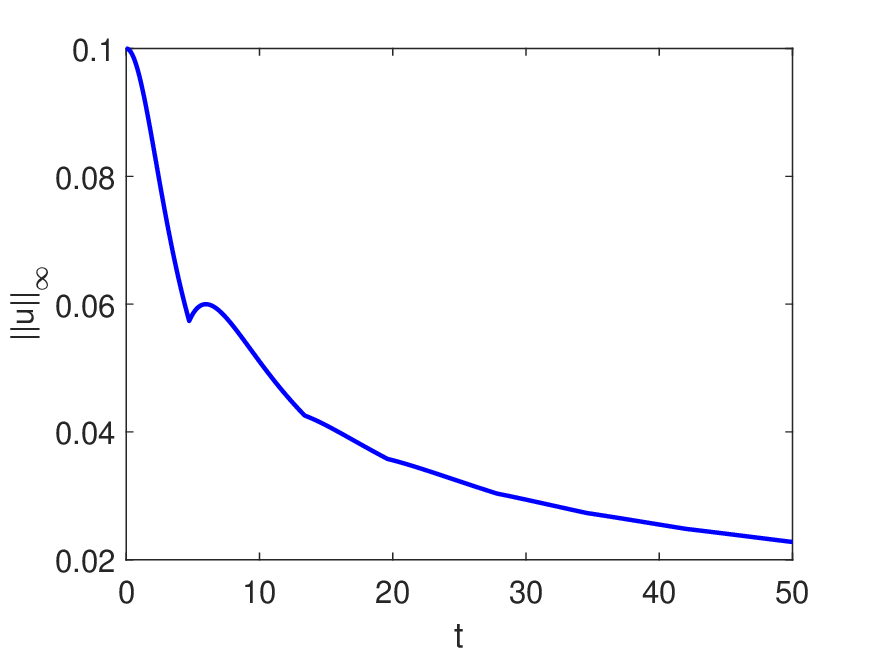}
 \end{minipage}
\center
 \hspace{-55pt} 
\begin{minipage}[t]{0.35\linewidth}
   \includegraphics[width=2.5in]{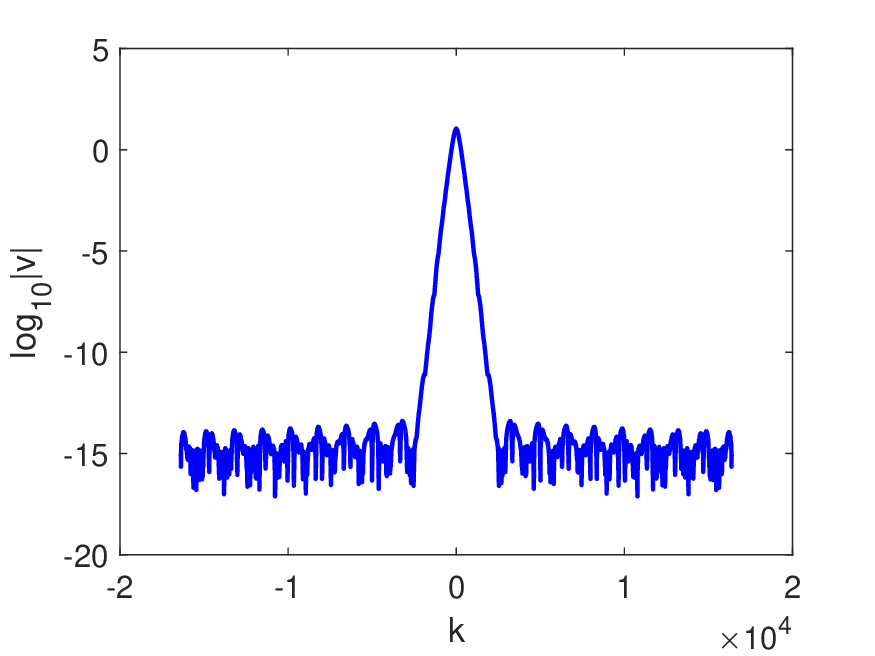}
 \end{minipage}
   \caption{ The profile of the solution to the fractional KdV-BBM equation \eqref{fKdVBBM} for $\alpha=0.9$ and $\delta=0.1$ at t=50 (upper left panel), variation in the $L_{\infty}$ norm of the corresponfding solution (upper right panel),  modulus of the Fourier coefficients for the corresponding solution  (bottom panel) }
  \label{a09}
\end{figure}
\end{center}  


Now, we concentrate on the case $0< \alpha \leq \frac{1}{3}$. In the third experiment, we choose  $\alpha=0.2$, $\delta=1.1$, the space interval $48[-\pi,\pi]$ with grid points $N=2^{16}$ and time step size $\Delta t = 10^{-4}$ until the time at which the relative energy stays below $5*10^{-3}$. Doing so, the code is stopped at time $t=11.527$. The solution profile at several time can be seen in Figure \ref{a03}. According to this numerical experiment, there is an indication of blow-up in the solution. 
\begin{center}
\begin{figure}[ht!]
 \begin{minipage}[t]{0.35\linewidth}
   \includegraphics[width=2.5in]{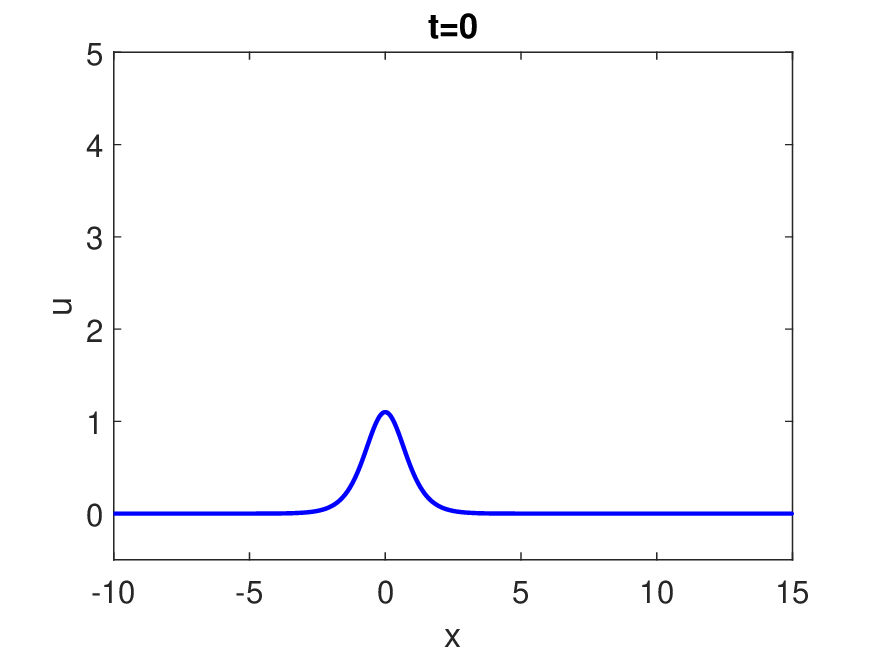}
 \end{minipage}
\hspace{30pt}
\begin{minipage}[t]{0.35\linewidth}
   \includegraphics[width=2.5in]{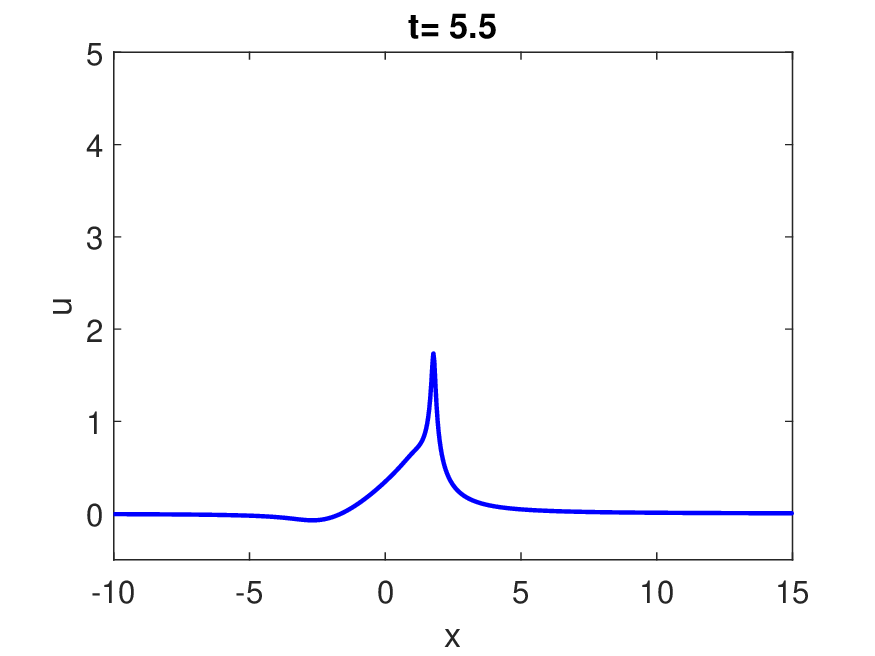}
 \end{minipage}

\begin{minipage}[t]{0.35\linewidth}
   \includegraphics[width=2.5in]{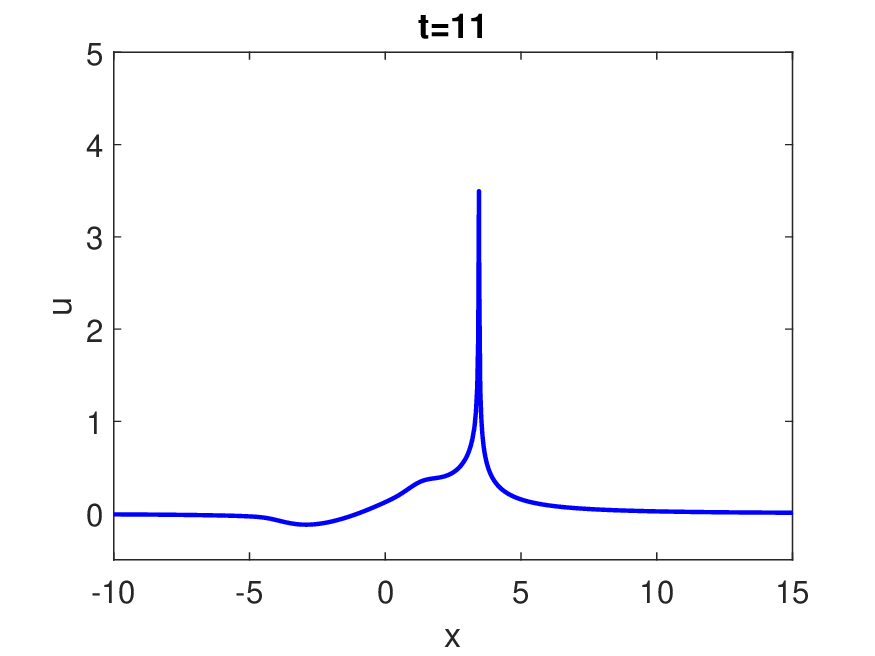}
 \end{minipage}
\hspace{30pt}
\begin{minipage}[t]{0.35\linewidth}
   \includegraphics[width=2.5in]{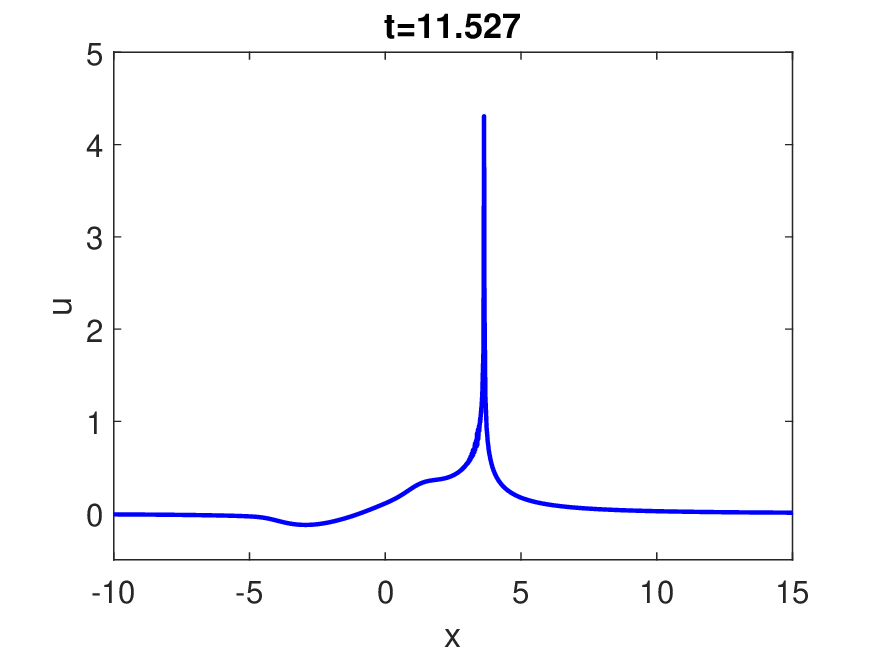}
 \end{minipage}
  \caption{The profile of the solution to the fractional KdV-BBM equation \eqref{fKdVBBM} for $\alpha=0.2$ and $\delta=1.1$ at several times}
 \label{a03}
\end{figure}
\end{center}
In the left panel of Figure \ref{blowuptime}, the variation of the $L_{\infty}$ norm of the corresponding solution is presented.  Although the initial guess is not sufficiently small as assumed in Theorem \ref{theorem}, numerical results indicate the existence of solutions until the time $\frac{1}{\delta^2}$ and after that time the solution blows up. In the right panel, we depict the resolution of the solution in the Fourier space at the final time.

\begin{figure}[ht!]
 \begin{minipage}[t]{0.35\linewidth}
   \includegraphics[width=2.5in]{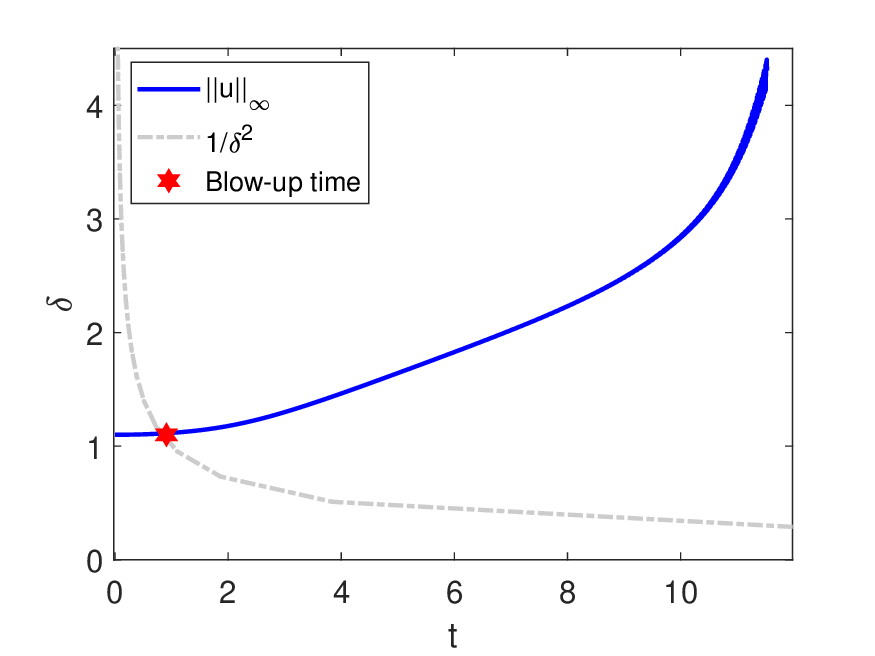}
 \end{minipage}
 \hspace{30pt}
 \begin{minipage}[t]{0.35\linewidth}
   \includegraphics[width=2.5in]{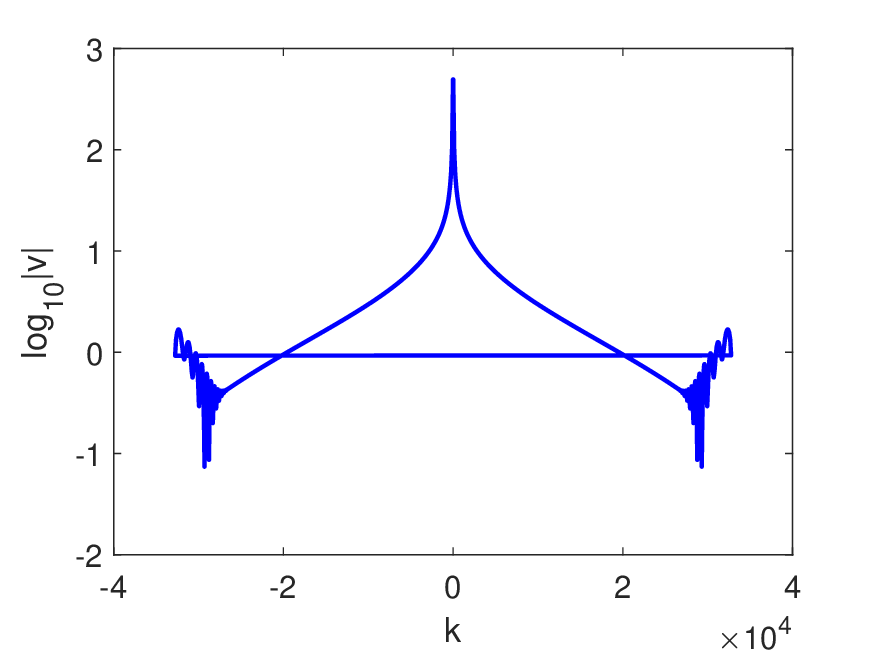}
 \end{minipage}
  \caption{The variation in the $L_{\infty}$ norm of the wave profile of the fractional KdV-BBM equation \eqref{fKdVBBM} for $\alpha=0.2$ and $\delta=1.1$ with time 
(left panel),   modulus of the Fourier coefficients for the corresponding solution  (right panel) }
  \label{blowuptime}
\end{figure}

In the last experiment, we choose $\alpha=0.2< \frac{1}{3}$ and small initial data such as $\delta=0.1.$ We consider the same time and space intervals also with the same discretization as in the second experiment (Figure \ref{a09}). We show the profile at final time in the upper left panel of Figure \ref{a02} and there is no indication of blow-up because the solution is clearly radiated away to infinity in this case. In the right panel of the same figure,  $L_{\infty}$ norm of the numerical solution decreases with time. The bottom panel shows the resolution of the solution in Fourier space.  Moreover, the energy is preserved up to order $10^{-12}$  in this experiment.
 
\begin{center}
\begin{figure}[ht!]
 \begin{minipage}[t]{0.35\linewidth}
   \includegraphics[width=2.5in]{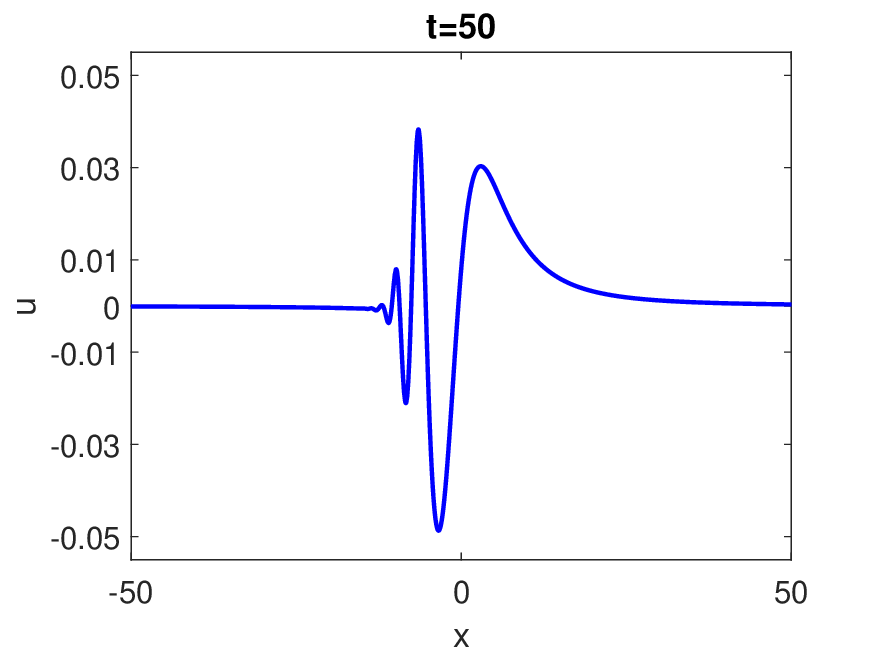}
 \end{minipage}
 \hspace{30pt}
 \begin{minipage}[t]{0.35\linewidth}
   \includegraphics[width=2.5in]{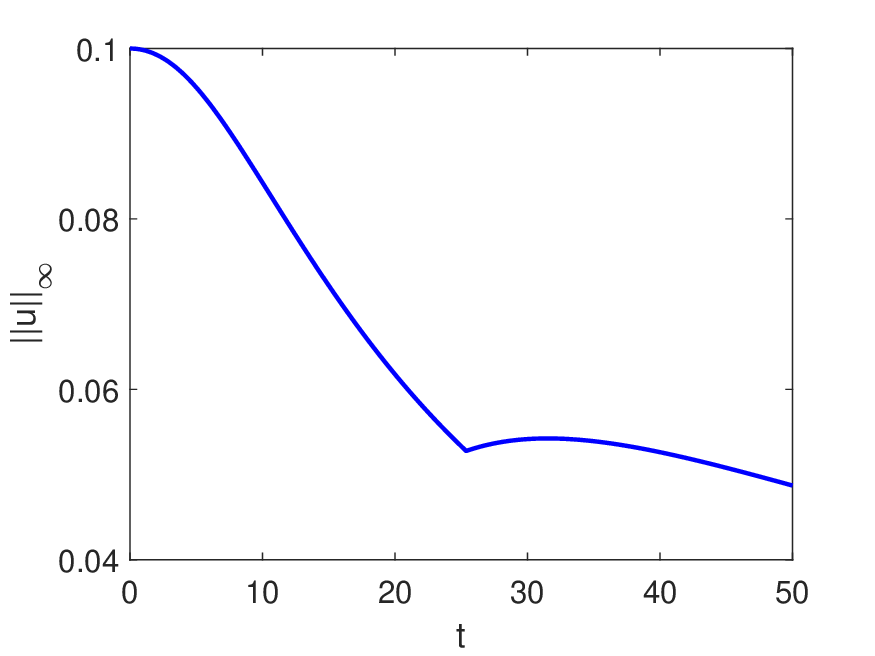}
 \end{minipage}
  \center
  \hspace{-55pt}
 \begin{minipage}[t]{0.35\linewidth}
   \includegraphics[width=2.5in]{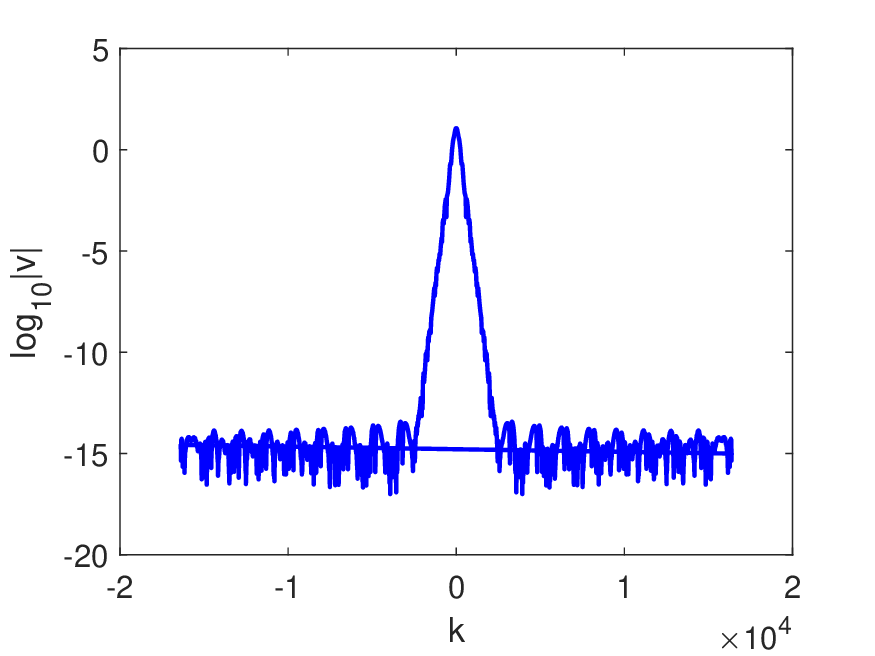}
 \end{minipage}
  \caption{The wave profile of the fractional KdV-BBM equation \eqref{fKdVBBM} for $\alpha=0.2$ and $\delta=0.1$ at time $t=50$  
(upper left panel),  variation in the $L_{\infty}$ norm of the corresponding solution with time (upper right panel) , modulus of the Fourier coefficients for the corresponding solution (bottom panel)}
  \label{a02}
\end{figure}
\end{center}

\indent We now summarize our numerical findings: Let us consider smooth initial data $u_0 \in L^2(\mathbb{R})$ with a single hump.\\\\
\indent (i) When  $\alpha>1/3$, the fractional KdV-BBM equation has smooth solutions for all $t$.\\\\
\indent (ii) When  $\alpha<1/3$, we have two scenarios. If $\delta>1$, then the solutions of the fractional KdV-BBM equation blow-up at finite time. On the other hand, we see that solutions stay smooth for $\delta < 1$.\\\\
Finally, we present our analytical and numerical results in Figure \ref{table}.

\begin{figure}[ht!]
\centering
\begin{tikzpicture}[scale=5.5][baseline=0pt]
 \draw[->] (0,0) -- (2.3,0) node[right]{$\delta$};
 \draw[->] (0,0) -- (0,1.3) node[left]{$\alpha$};
 \draw  (-0.1,0)  node[below]{$0$};
 \draw  (1,0)  node[below]{$1$};
 \draw  (2,0)  node[below]{$2$};
 \draw  (0,0.33)  node[left]{$1/3$};
 \draw  (0,1)  node[left]{$1$};

\path [fill=yellow!25] (0.26,0) rectangle (2.2,1);
\path [fill=red!20!] (0.005,0) rectangle (0.29,1);
\path [fill=lightgray!30] (1,0) rectangle (2.2,0.33);


\draw[dashed] (0,1)--(2.2,1);
\draw[dashed] (0,0.33)--(2.2,0.33);

\draw[dashed] (1,0)--(1,1);


\node [above] at (1.6,0.2) {\small{blow-up}};
\node [below] at (1.6,0.2) {\small{(numerical observation)}};

\node [above] at (1.6,0.65) {\small{smooth solutions}};
\node [below] at (1.6,0.65) {\small{(numerical observation)}};

\node [above] at (0.64,0.2) {\small{smooth solutions}};
\node [below] at (0.64,0.2) {\small{(numerical observation)}};

\node [above] at (0.64,0.65) {\small{smooth solutions}};
\node [below] at (0.64,0.65) {\small{(numerical observation)}};

\node [above] at (0.15,0.005) {\rotatebox{90} {\small{Local existence by Theorem \ref{theorem}}}};

\end{tikzpicture}
\caption{The cases for the existence of solutions to the fractional KdV-BBM equation} \label{table}
\end{figure}
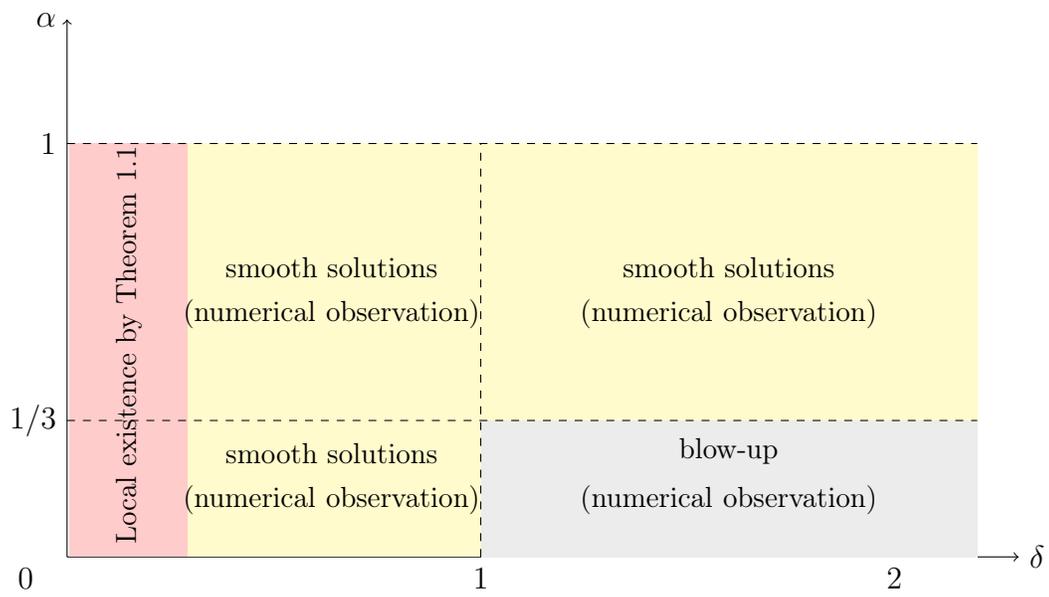

\section{Results and Discussion}
In this paper, the fractional type KdV-BBM equation is derived as a pyhsical model. The long time existence result for the Cauchy problem \eqref{fKdVBBM}-\eqref{u0} is established. The maximal existence time is extended from $\frac{1}{\epsilon}$ to $\frac{1}{\epsilon^2}$ by using a modified energy technique. We use a Fourier pseudospectral method for the same problem. Numerical examples indicate that the solutions exist globally in time for $\alpha>\frac{1}{3}$ (no blow-up) and radiated and decompose into soliton for large time. For $\alpha < \frac{1}{3}$, the solutions with sufficiently large initial data appear to have a blow-up whereas the solutions with small initial data are global. In accordance with our main theorem, it is also numerically observed that the lifespan of solutions can be enhanced, when smaller initial data are considered. Although the equation \eqref{fKdVBBM} has dispersive terms of both fractional KdV and fractional BBM equations, the obtained analytical and numerical results are similar to what has been found for the solutions to the fractional BBM equation in \cite{nilsson} and \cite{klein}.



\clearpage
\noindent \textbf{Acknowledgements}\\
The author was supported by the \emph{Federal Ministry for Economic Affairs and Energy of Germany} (BMWE -- Bundesministerium für Wirtschaft und Energie der Bundesrepublik Deutschland). 
The work reported in this paper was partially developed during the stay of the author at the University of Burgundy, in the academic year 2021/2022. The author wishes to thank Christian Klein for making this visit possible.
The author would also like to thank three anonymous referees for the comments and suggestions that helped to improve the manuscript.
\bibliographystyle{vancouver}
\bibliography{kaynak}
\end{document}